\documentclass{amsart}

\usepackage{srollens-en}

\renewcommand{\lg}{\ensuremath{\mathfrak g}}
\newcommand{\lh}{\mathfrak h}
\newcommand{\einsnull}[1]{{{#1}^{1,0}}}
\newcommand{\nulleins}[1]{{{#1}^{0,1}}}

\newcommand{\I}{{i}}

\title[Dolbeault cohomology of nilmanifolds]{Dolbeault cohomology of nilmanifolds with left-invariant complex structure}

\author{S\"onke Rollenske}

\address{Dr. S\"onke Rollenske\\
Mathematisches Institut \\							
 Rheinische Fried\-rich{\-}-Wil\-helms-Uni\-versi\-t\"at Bonn \\
 Endenicher Allee 60  \\
 53115 Bonn, 
Germany}
\email{srollens@math.uni-bonn.de}

\begin{document}

\begin{abstract}
We discuss the known evidence for the conjecture that the Dolbeault cohomology of nilmanifolds with left-invariant complex structure can be computed as Lie-algebra cohomology and also mention some applications. 
\end{abstract}

\maketitle

\section{Introduction}

Dolbeault cohomology is one of the most fundamental holomorphic invariants of a complex manifold $X$ but in general it is quite hard to compute. If $X$ is K\"ahler then this amounts to describing the decomposition of the de Rham cohomology
\[H^k_{dR}(X,\IC)=\bigoplus_{p+q=k} H^{p,q}(X) =\bigoplus_{p+q=k}H^q(X, \Omega^p_X)\]
but in general there is only a spectral sequence connecting these invariants.

One case where at least de Rham cohomology is easily computable is the case of nilmanifolds, that is, compact quotients of real nilpotent Lie groups. If $M=\Gamma\backslash G$ is a nilmanifold and $\lg$ is the associated nilpotent Lie algebra Nomizu proved that we have a natural isomorphism
\[ H^*(\lg, \IR) \isom H^*_{\mathrm{dR}}(M, \IR)\]
where the left hand side is the Lie-algebra cohomology of $\lg$. In other words, computing the cohomology of $M$ has become a matter of linear algebra

There is a natural way to endow an even-dimensional nilmanifold with an almost complex structure: choose any endomorphism $J:\lg\to \lg$ with $J^2=-\id$ and extend it to an endomorphism of $TG$, also denoted by $J$, by left-multiplication. Then $J$ is invariant under the action of $\Gamma$ and descends to an almost complex structure on $M$. If $J$ satisfies the integrability condition
\begin{equation}\label{nijenhuis}
 [x,y]-[Jx,Jy]+J[Jx,y]+J[x,Jy]=0 \text{ for all } x,y \in \lg
\end{equation}
then, by Newlander--Nirenberg \cite[p.145]{Kob-NumII}, it makes $M_J=(M,J)$ into a complex manifold.

In this survey we want to discuss the conjecture
\begin{center}
\emph{ The Dolbeault cohomology of a nilmanifold with left-invariant complex structure $M_J$ can be computed using only left-invariant forms.} 
\end{center}
This was stated as a question in \cite{cfgu00, con-fin01} but we decided to call it Conjecture in the hope that it  should motivate other people to come up with a proof or a counterexample.
A more precise formulation in terms of Lie-algebra cohomology is given in Section \ref{reminder}. 

Before concentrating on this topic we would like to indicate why nilmanifolds have attracted much interest over the last years. Their main feature is that the construction and study of left-invariant geometric structures on them usually boils down to  finite dimensional linear algebra. On the other hand, the structure is sufficiently flexible to allow the construction of many exotic examples. We only want to mention the three most prominent in complex geometry:
\begin{itemize}
\item If $G$ is abelian then $M_J$ is a complex torus. 
\item The Iwasawa manifold $X=\Gamma\backslash G$ is obtained as the quotient of the complex Lie group
 \[G=\left\{ \begin{pmatrix}1 & z_1 &z_3\\ 0&1 &z_2\\ 0&0&1\end{pmatrix}\right\}\subset \mathrm{Gl}(3,\IC)\]
by the lattice $\Gamma=G\cap \mathrm{Gl}(3,\IZ[\I])$ and as such is complex parallelisable. Nakamura studied its small deformations and thus showed that a small deformation of a complex parallelisable manifold need not be complex parallelisable \cite{nakamura75}.

Observe that $X$ cannot be K\"ahler since $dz_3-z_2dz_1$ is a holomorphic 1-form that is not closed.

\item Kodaira surfaces, also known as Kodaira-Thurston manifolds, had appeared in Kodaira's classification of compact complex surfaces as non-trivial principal bundle of elliptic curves over an elliptic curve \cite{kodaira66} and were later considered independently by Thurston as the first example of a manifold that admits both a symplectic and a complex structure but no K\"ahler structure. In our context it can be described as follows: let
\[G=\left\{ \begin{pmatrix}1 & \bar z_1 &z_2\\ 0&1 & z_1\\ 0&0&1\end{pmatrix}\mid z_1, z_2\in \IC\right\}\subset \mathrm{Gl}(3,\IC)\]
and $\Gamma=G\cap \mathrm{Gl}(3,\IZ[\I])$. Then $G\isom \IC^2$ with coordinates $z_1, z_2$ and the action of $\Gamma$ on the left is holomorphic; the quotient is a compact complex manifold. If we set $\alpha= dz_1\wedge(d\bar z_2- z_1 d\bar z_1)$ then $\alpha+\bar \alpha$ is a  left-invariant symplectic form on $G$ and thus descends to the quotient.
\end{itemize}
In fact, the first example is the only nilmanifold that can admit a K\"ahler structure \cite{ben-gor88}, so none of the familiar techniques available for K\"ahler manifolds will be useful in our case.

Some more applications in complex geometry will be given in Section \ref{apps}. Nilmanifolds also play a role in hermitian geometry \cite{ags01, bdv09, lauret06}, riemannian geometry \cite{gromov78, buser-karcher81}, ergodic theory \cite{host-kra05}, arithmetic combinatorics \cite{green-tao06}, 
and theoretical physics \cite{MR2542937, gmpt07}.

In order to discuss the above conjecture on Dolbeault cohomology we start by sketching the proof of Nomizu's theorem because some of the ideas carry over to the holomorphic setting. Then we recall the necessary details on Dolbeault cohomology to give a precise statement of the conjecture.
It turns out that we are in a good position to prove the conjecture whenever we can inductively decompose the nilmanifold with left-invariant complex structure into simpler pieces. This is due to Console and Fino \cite{con-fin01}, generalising previous results of Cordero, Fern\'andez, Gray and Ugarte \cite{cfgu00}. 

Section \ref{new} contains the only new result in this article. We prove that the conjecture always holds true if we pass to a suitable quotient of the nilmanifold with left-invariant complex structure and also discuss some possible approaches to attack the general case.

\subsection{Notations}
Throughout the paper $G$ will be a simply connected nilpotent real Lie-group with Lie-algebra $\lg$. Every nilpotent Lie group can be realised as a subgroup of the group of upper triangular matrices with 1's on the diagonal.

 We will always assume that $G$ contains a lattice $\Gamma$ thus giving rise to a (compact) nilmanifold $M=\Gamma\backslash G$. Elements in $\lg$ will usually be interpreted as left-invariant vector fields on $G$ or on $M$.
We restrict our attention to those complex structures on  $M$ that are induced by an  integrable left-invariant complex structure on $G$ and are thus uniquely determined by an (integrable) complex structure  $J:\lg\to \lg$. The resulting complex manifold is denoted $M_J$. Note that even on a real torus of even dimension at least 6 there are many complex structures that do not arise in this way \cite{catanese02}.

The group $G$ is determined up to isomorphism by the fundamental group of $M$ \cite[Corollary 2.8, p.45]{VinGorbShvart} and by abuse of notation we sometimes call $\lg$ the Lie-algebra of $M$.

\section{Real nilmanifolds and Nomizu's result on de Rham cohomology}
The aim of this section is to prove Nomizu's theorem.
\begin{theo}[{Nomizu \cite{nomizu54}}]\label{nomizu}
 Let $M$ be a compact nilmanifold. Then the inclusion of left-invariant differential forms in the de Rham complex
\[\Lambda^\bullet \lg^*\into \ka^\bullet(M)\] induces an isomorphism between the Lie-algebra cohomology of $\lg$ and the de Rham cohomology of $M$, 
\[ H^*(\lg, \IR) \isom H^*_{\mathrm{dR}}(M, \IR).\]
\end{theo}
Since some of the main results on Dolbeault cohomology discussed in the next section rely on similar ideas we will examine the proof in some detail: at its heart lies an inductive argument.

Let $M=\Gamma\backslash G$ be a real nilmanifold with associated Lie algebra $\lg$ and let $\kz G$ be the centre of $G$. By \cite[p. 208]{Cor-Green}, $\kz\Gamma=\Gamma\cap \kz G$ is again a lattice and  the projection $G\to G/\kz G$ descends to a fibration $M\to M'$. The fibres are real tori $T=\kz G/\kz\Gamma$. Since elements in $\kz G$ commute with elements in $\Gamma$ their action descends to the quotient and $M\to M'$ is   a principal $T$-bundle.

To iterate this process we recall the following definition.
 \begin{defin}\label{ZgCg}
For a Lie-algebra $\lg$ we call
\[\kz^0\lg:= 0, \qquad \kz^{i+1} \lg := \{ x\in \lg \mid [x,\lg ]\subset  \kz^{i}\lg\}\]
the ascending central series and 
\[\kc^0\lg :=\lg, \qquad \kc^{i+1}\lg := [\kc^{i}\lg, \lg]\]
the descending central series of $\lg$.

The Lie-algebra is called nilpotent if there is a $\nu\in \IN$ such that $\kz^\nu\lg=\lg$, or equivalently $\kc^\nu\lg=0$.
 The minimal such $\nu=\nu(\lg)$ is called the index of nilpotency or step-length of $\lg$.
\end{defin}
The same definition can be made on the level of the Lie-group $G$ and the resulting sub-algebras and subgroups correspond to each other under the exponential map.

Proceeding inductively, we can use the first filtration on $\lg$  to decompose  $M$ geometrically; the second one induces a similar decomposition since $\kc^i\lg\subset \kz^{\nu-i}\lg$. More precisely, if we denote by $T_i$ the torus obtained as a quotient of $\kz^i G /\kz^{i+1} G$ by $\kz^i \Gamma /\kz^{i+1} \Gamma$ then there is a tower
\begin{equation}\label{tower}
{\xymatrix{ T_1 \ar@{^(->}[r] & M_1 \ar[d]^{\pi_1}\\
T_2 \ar@{^(->}[r] & M_2\ar[d]^{\pi_2}\\
&\vdots\ar[d]\\
T_{\nu-1}\ar@{^(->}[r] & M_{\nu-1}\ar[d]^{\pi_{\nu-1}}\\
& M_\nu} }
\end{equation}
and each $\pi_i:M_i\to M_{i+1}$ is a $T_i$-principal bundle.

This geometric description is crucial in the proof of  Nomizu's Theorem. The underlying idea is quite simple: we perform  induction over the index of nilpotency $\nu$. If $\nu=1$, i.e., $\lg$ is abelian, then $M$ is a torus and the result is well known. For the induction step, we consider $M$ as a principal torus bundle over a
nilmanifold $M'$ with lower nilpotency index. Then we have to combine our knowledge of the cohomology of the fibre and of the base to describe the cohomology of the total space $M$. This is achieved by  means of two spectral sequences, the Leray-Serre spectral sequence and the Serre-Hochschild spectral sequence.

Let us work this out a bit more in detail starting on the geometric side: let $\ka^k(M)$ be the the space of smooth differential $k$-forms on $M$ and consider the de Rham complex
\begin{equation*}
 0\to \ka^0(M)\overset{d}{\longrightarrow}\ka^1(M)\overset{d}{\longrightarrow}\dots \overset{d}{\longrightarrow}\ka^{n}(M)\to 0.
\end{equation*}

The principal bundle $\pi: M\to M'$ with fibre $T$ induces an inclusion $\pi^*\ka^1(M')\into \ka^1(M)$ and thus a filtration of $\ka^k(M)$ whose graded pieces are generated by forms of the type $(\pi^*\alpha)\wedge \beta$ where $\beta$ is a differential form along the fibres.
Decomposing also the differential and starting with the vertical component we have constructed a version of the Leray Serre spectral sequence
\[ E_2^{p,q}=H^p(M', H^q(T, \IR))\implies H^{p+q}_{dR}(M).\]
In the general case the $E_2$-term has to be interpreted as cohomology with values in a local system but since we have a principal bundle with connected structure group the monodromy action on $H^q(T, \IR)$ is trivial and we have 
$E_2^{p,q}=H^p_{dR}(M')\tensor H^q_{dR}(T)$.

Now we repeat the construction on the level of left-invariant forms. Consider $\Lambda^\bullet \lg^*$ as a subcomplex of the de Rham complex $(\ka^\bullet, d)$. The differential of a $k$-form $\alpha$ can be defined entirely in terms of the Lie-bracket and the Lie-derivative as
\begin{multline*}\label{ch-diff}
(d_k\alpha)(x_1, \dots , x_{k+1}):=\sum_{i=1}^{k+1} (-1)^{i+1} x_i (\alpha(x_1, \dots ,\hat x_i, \dots , x_{k+1}))\\
+\sum_{1\leq i <j\leq k+1} (-1)^{i+j} \alpha([x_i,x_j], x_1, \dots, \hat x_i, \dots, \hat x_j,\dots , x_{k+1}).
\end{multline*}
For left-invariant $\alpha\in \Lambda^k\lg^*$ and  $x_i\in \lg$ it reduces to 
\begin{equation*} 
(d_k\alpha)(x_1, \dots , x_{k+1})=\sum_{1\leq i <j\leq k+1} (-1)^{i+j} \alpha([x_i,x_j], x_1, \dots, \hat x_i, \dots, \hat x_j,\dots , x_{k+1})
\end{equation*}
and the  complex $(\Lambda^\bullet \lg^*, d)$ is defined purly algebraically. It is known as Chevalley complex \cite{che-eil48} and computes the Lie-algebra cohomology of $\lg$ (see also \cite[Chapter 7]{Weibel}).

If the fibration $\pi:M\to M'$ corresponds to  the short exact sequence
\[0\to \lh\to \lg\to \lg/\lh\to 0\]
where $\lh=\kz\lg$ as explained above then the dual sequence induces a filtration on the exterior powers $\Lambda^k\lg^*$ and we can organise the graded pieces into a spectral sequence, the Hochschild-Serre spectral sequence (see \cite[Section 7.5]{Weibel}), with 
\begin{gather*}
 E_0^{p,q}=\Lambda^p(\lg/\lh)^*\tensor \Lambda^q\lh^*\\
E_2^{p,q}=H^p(\lg/\lh, H^q(\lh))=H^p(\lg/\lh)\tensor H^q(\lh)\implies H^{p+q}(\lg, \IR).
\end{gather*}
The second description of the $E_2$-term holds in our setting since $\lh$ is contained in the centre of $\lg$, which corresponds to $\pi$ being a principal bundle.

Now we deduce a proof of Nomizu's theorem: we know the result for the torus and then proceed by induction on the nilpotency index. The inclusion $(\Lambda^\bullet\lg^*,d)\into (\ka^\bullet(M), d)$ is compatible with the filtrations we introduced and thus we get an induced homomorphism of spectral sequences. At the $E_2$ level this is
\[ H^p(\lg/\lh)\tensor H^q(\lh)\to H^p_{dR}(M')\tensor H^q_{dR}(T)\]
which is an isomorphism by induction hypothesis. Thus also in the limit we have the desired isomorphism
\[ H^*(\lg)\overset{\isom}{\longrightarrow}H^*_{dR}(M).\]

\begin{rem}
The statement we just proved extends to solvmanifolds, i.e., compact quotients of solvable groups, that satisfy the so-called Mostow condition 
\cite{mostow61}.  The de Rham cohomology of more general solvmanifolds can be studied via an auxiliary construction due to Guan \cite{guan07} which was recently reconsidered by Console and Fino \cite{con-fin09}.
\end{rem}

\section{Left-invariant complex structures and Dolbeault cohomology}

We start this section by recalling the definition of Dolbeault cohomology and giving the precise statement of the conjecture. Then we discuss to what extent the proof of Nomizu's result, discussed in the preceding section, carries over to the holomorphic setting. After mentioning the openness result of Console and Fino we will also give some new results and discuss directions of future research.

\subsection{Reminder on Dolbeault cohomology}\label{reminder}
Recall that an (integrable) complex structure on a differentiable manifold $M$ is a vector bundle endomorphism $J$ of the tangent bundle which satisfies $J^2=-\id$ and  the integrability condition \eqref{nijenhuis}. The endomorphism $J$ induces a decomposition of the complexified tangent bundle by letting pointwise $\einsnull TM\subset T_\IC M=T M\tensor\IC$ be the $\I$-eigenspace of $J$. Then the $-\I$-eigenspace is $\nulleins TM=\overline{\einsnull TM}$. Note that $\einsnull TM $ is naturally isomorphic to $(TM,J)$ as a complex vector bundle via the projection, and the integrability condition can be formulated as $[\einsnull TM, \einsnull TM]\subset \einsnull TM$.

The bundle of differential $k$-forms decomposes
\[\Lambda^kT^*_\IC M=\bigoplus_{p+q=k}\Lambda^p \einsnull{T^*}M\tensor \Lambda^q\nulleins{T^*}M=\bigoplus_{p+q=k}\Lambda^{p,q}T^*M,\]
and we denote by $\ka^{p,q}(M)$ the $\kc^\infty$-sections of the bundle $\Lambda^{p,q}T^*M$, i.e., the global differential forms of type $(p,q)$.

The integrability condition \eqref{nijenhuis} is equivalent to the decomposition of the differential $d=\del+\delbar$  and for all $p$ we get  the Dolbeault complex
\[(\ka^{p,\bullet}(M_J), \delbar): 0\to \ka^{p,0}(M)\overset\delbar\longrightarrow \ka^{p,1}(M)\overset\delbar\longrightarrow \dots\]
The Dolbeault cohomology groups $H^{p,q}(M)=H^q(\ka^{p,\bullet}(M), \delbar)$ are one of the most fundamental holomorphic invariants of $M_J$; from another point of view, the  Dolbeault complex computes the  cohomology groups of  the sheaf $\Omega^p_{M_J}$ of holomorphic $p$-forms.

In case $M$ is a nilmanifold and $J$ is left-invariant all of the above can be considered at the level of left-invariant forms. Decomposing $\lg^*_\IC=\einsnull{\lg^*}\oplus \nulleins{\lg^*}$ and setting $\Lambda^{p,q}\lg^*=\Lambda^p\einsnull{\lg^*}\tensor \Lambda^q\nulleins{\lg^*}$ we get subcomplexes
\begin{equation}\label{inclusion}
 (\Lambda^{p,\bullet}\lg^*, \delbar)\into (\ka^{p,\bullet}(M_J), \delbar). 
\end{equation}
In fact, the left hand side has a purely algebraic interpretation worked out in \cite{rollenske09a}: $\nulleins\lg$ is a Lie-subalgebra of $\lg_\IC$ and the adjoint action followed by the projection to the $(1,0)$-part makes $\einsnull\lg$ into an $\nulleins\lg$-module. Then the complex $(\Lambda^{p,\bullet}\lg^*, \delbar)$ computes the Lie-algebra cohomology of $\nulleins\lg$ with values in $\Lambda^p\einsnull{\lg^*}$ and we call
\[H^{p,q}(\lg, J)=H^q(\nulleins\lg, \Lambda^p\einsnull{\lg^*})=H^q(\Lambda^{p,\bullet}\lg^*, \delbar)\]
the Lie-algebra Dolbeault cohomology of $(\lg, J)$. 

We can now formulate the analogue of Nomizu's theorem for Dolbeault cohomology as a conjecture.
\begin{custom}[Conjecture]\label{conj}
 Let $M_J$ be a nilmanifold with left-invariant complex structure. Then the map
\begin{equation*}
\phi_J: H^{p,q}(\lg, J)\to H^{p,q}(M_J)
\end{equation*}
induced by \eqref{inclusion} is an isomorphism.
\end{custom}
It is known that $\phi_J$ is always injective (see \cite{con-fin01} or \cite{rollenske09a}). 

We will accumulate evidence for the conjecture over the next sections and also explain which are the open cases.

\subsection{The inductive proof}

In order to extend the idea of Nomizu's proof to Dolbeault cohomology we need to have three ingredients:
\begin{enumerate}
 \item Can we start the induction, i.e., can we express the Dolbeault cohomology of a complex torus as a suitable Lie-algebra cohomology?
\item Does the complex geometry of nilmanifolds allow us to proceed by induction? For example, is every nilmanifold with left-invariant complex structure a holomorphic principal bundle?
\item Are there spectral sequences that play the role of the Leray-Serre and Hochschild-Serre spectral sequence for (Lie-algebra) Dolbeault cohomology?
\end{enumerate}
It is well known that the first question has a positive answer (see e.g. \cite[p.15]{Birkenhake-Lange}). In our language, assume that $\lg$ is abelian and $J$ is a complex structure. Then  the differential in the Lie-algebra Dolbeault complex $(\Lambda^{p, \bullet}\lg^*, \delbar)$   is trivial (being induced by the adjoint action) and thus 
\[H^{p,q}(\lg, J)=\Lambda^{p,q}\lg^*=\Lambda^p\lg^*\tensor \Lambda^q\bar\lg^*=H^{p,q}(M_J).\]

Unfortunately, the answer to the second question is negative. We will discuss the geometry of nilmanifolds with left-invariant complex structure in Section \ref{geometry} and see that nevertheless the inductive approach works in many important special cases.

The positive answer to the third question, important for the induction step, has been worked out by Cordero, Fern\'andez, Gray and Ugarte \cite{cfgu00} for principal holomorphic torus bundles and in greater generality by Console and Fino \cite{con-fin01}. The extra grading coming from the $(p,q)$-type of the differential forms makes the notation and the construction of the necessary spectral sequences more involved. For the usual Dolbeault cohomology of a holomorphic fibration the result goes back to Borel \cite[Appendix II, Theorem 2.1]{Hirzebruch}.

\begin{prop}[Console, Fino]\label{indstep}
 Let $M_J$ be a nilmanifold with left-invariant complex structure and let $\pi:M\to M'$ be a holomorphic fibration with typical fibre $F$ induced by a  $\Gamma$-rational and $J$-invariant ideal $\lh\subset \lg$ (as explained in Section \ref{geometry}). If for all $p,q$ we have  
\[H^{p,q}(\lh, J\restr\lh)\isom H^{p,q}(F)\text{ and }H^{p,q}(\lg/\lh, J')\isom H^{p,q}(M'),\]
 where $J'$ is the complex structure on $\lg/\lh$ induced by $J$, then also \[H^{p,q}(\lg, J)\isom H^{p,q}(M).\]
\end{prop}
Clearly, with the above proposition we can proceed inductively to compute the Dolbeault cohomology of iterated holomorphic principal bundles as we did in the real case.  Unfortunately, considering principal holomorphic torus bundles is not enough so we really need to decide when a nilmanifolds with left-invariant complex structure admits a suitable fibration.

\subsubsection{When is a nilmanifold with left-invariant complex structure an iterated (principal) bundle?}\label{geometry}
We have seen that we need to understand the geometry of nilmanifolds with left-invariant complex structure, in particular whether  there are natural fibrations over nilmanifolds of smaller dimension. In general, the projections in  the tower of (real) principal bundles \eqref{tower} will not be holomorphic, for example, the centre could be odd-dimensional.

It would be convenient if we could detect fibrations of $M$ by studying only the Lie-algebra $\lg$. For universal cover, i.e., the simply connected Lie group, this is easy: a fibration $G\to G'$ over another simply connected nilpotent Lie-group corresponds to a short exacts sequence of Lie algebras
\[ 0\to \lh\to \lg\to \lg'\to 0\]
or, in other words, to an ideal $\lh\subset \lg$. Here we use that, by the Baker-Campell-Hausdorff formula (see e.g. \cite[Section B.4]{Knapp}), the exponential map $\exp:\lg \to G$ is a diffeomorphism and hence every ideal induces a closed subgroup of $G$.

If we  look at a  2-dimensional torus $M=\IR^2/\IZ^2$ then every  1-dimensional subspace $\lh$ in the abelian Lie-algebra $\lg=\IR^2$ is an ideal. But there is some extra structure: a basis for the lattice (or, strictly speaking,  the logarithm of this basis) generates a $\IQ$-vector space $\lg_\IQ\isom\IQ^2\subset \lg$ such that $\lg_\IQ\tensor\IR=\lg$. Clearly, a 1-dimensional subgroup corresponding to $\lh\subset \lg$ closes to a circle in the quotient if an only if it has rational slope, i.e., if and only if  $\lh\cap \lg_\IQ$ is a $\IQ$-vector space of dimension 1.

The general case is captured in the following definition.
\begin{defin}
Let $\lg$ be a nilpotent Lie-algebra. A \emph{rational structure} for $\lg$ is a subalgebra $\lg_\IQ$ defined over $\IQ$ such that $\lg_\IQ\tensor \IR =\lg$.

A subalgebra $\lh\subset \lg$ is said to be rational with respect to a given rational structure $\lg_\IQ$ if $\lh_\IQ:=\lh\cap \lg_\IQ$ is a rational structure for $\lh$.

If $\Gamma$ is a lattice in the corresponding simply connected Lie-group $G$ then its associated rational structure is given by the $\IQ$-span of $\log \Gamma$. A rational subspace with respect to this structure is called \emph{$\Gamma$-rational}.
\end{defin}
\begin{rem}\label{Qstr}
One has to check that this is well defined, i.e., that the $\IQ$-span of $\log\Gamma$   gives a rational structure. Indeed more is true: a nilpotent Lie-algebra admits a $\IQ$-structure if and only if the corresponding simply connected Lie-group contains a lattice \cite[Theorem 5.1.8]{Cor-Green}.

This criterion makes it particularly simple to produce examples: given a nilpotent Lie-algebra $\lg$ with rational structure constants we know that there exists a lattice $\Gamma$ in the corresponding Lie-group $G$ and we get a compact nilmanifold $M=\Gamma\backslash G$. Since most properties of $M$ are encoded in $\lg$ there is usually no need to specify the lattice concretely. 
\end{rem}

Coming back to the original problem we have \cite[Lemma 5.1.4, Theorem 5.1.11]{Cor-Green}:
\begin{lem}
Let $\lh\subset \lg$ be an ideal. Then the fibration $G\to G/\exp\lh$ descends to a fibration of compact nilmanifolds $\pi:M\to M'$ if and only if $\lh$ is $\Gamma$-rational.
\end{lem}

In principle, all subspaces that are naturally associated to the Lie-algebra structure of $\lg$ are rational with respect to any rational structure in $\lg$. In particular this holds for the subspaces in the ascending  and descending central series (Definition \ref{ZgCg}) and intersections thereof \cite[p. 208]{Cor-Green}.

If we add left-invariant complex structures,  we would like the fibration $\pi:M_J\to M'_{J'}$ to be holomorphic as well, which, by left-invariance, is the same as to say that $\lg\to \lg'$ is complex linear or equivalently that $\lh$ is a complex subspace of $(\lg, J)$. We have proved
\begin{prop}\label{fibration}
 Let $M_J$ be a nilmanifold with left-invariant complex structure. Then $\lh\subset \lg$ defines a holomorphic fibration $\pi:M_J\to M'_{J'}$ if and only if $\lh$ is a $J$-invariant and $\Gamma$-rational ideal in $\lg$.
\end{prop}

It is time for an example that shows what can go wrong:
\begin{exam}\label{badex}
We define a 6-dimensional Lie algebra $\lh_7$ with basis $e_1, \dots, e_6$ where, up to anti-commutativity, the only non-zero brackets are
\[[e_1, e_2]=-e_4,\, [e_1, e_3]=-e_5,\, [e_2, e_3]=-e_6.\] 
   The vectors $e_4\dots,e_6$ span the centre $\kz^1\lh_7=\kc^1\lh_7$. 

Since the structure equations are rational there is a lattice $\Gamma$ in the corresponding simply connected Lie-group $H_7$ and we can consider the nilmanifold $M=\Gamma\backslash H_7$.

For $\lambda\in \IR$  we give a left-invariant complex structure $J_\lambda$ on $M$ by specifying a basis for the space of $(1,0)$-vectors:
\[(\einsnull{\lh_7})_\lambda:=\langle X_1=e_1-ie_2, X_2^\lambda= e_3-i (e_4-\lambda e_1), X_3^\lambda=-e_5+\lambda e_4+ie_6\rangle\]
One can check that $[X_1, X_2^\lambda]=X_3^\lambda$ and, since $X_3^\lambda$ is contained in the centre, the complex structure is integrable. The largest complex subspace of the centre is spanned by the real and imaginary part of $X_3^\lambda $ since the centre has real dimension three.

The simply connected Lie-group $H_7$ has  a filtration by subgroups induced by the filtration 
\[\lh_7\supset V_1=\langle \lambda e_2+e_3, e_4,Im(X_3^\lambda),Re(X_3^\lambda)\rangle\supset V_2=\langle Im(X_3^\lambda),Re(X_3^\lambda)\rangle \supset 0\]
on the Lie-algebra and, since all these are $J$ invariant, $H_7$ has the structure of a tower of principal holomorphic bundles with fibre $\IC$.
In fact, using the results of \cite{ugarte07}, a simple calculation shows that every complex structure on $\lh_7$ is equivalent to $J_0$.

Now we take the compatibility with  the lattice into account. 
The rational structure induced by $\Gamma$ coincides with the $\IQ$-algebra generated by the basis vectors $e_k$ and, by the criterion in Proposition \ref{fibration}, the fibrations on $H_7$ descends to the compact nilmanifold $M$ if and only if  $\lambda$ is rational.
In fact, one can check that for $\lambda\notin \IQ$ the Lie-algebra $\lh_7$ does not contain any non-trivial $J$-invariant and $\Gamma$-rational ideals, so there is no holomorphic fibration at all over a nilmanifold of smaller dimension.
\end{exam}

To understand when there is a suitable tower of fibrations on a nilmanifold the following definitions turn out to be useful:
\begin{defin}\label{stableseries}
Let $\lg$ be a nilpotent Lie-algebra with rational structure $\lg_\IQ$. We call an ascending filtration 
\[0=\ks^0\lg\subset \ks^1\lg\subset \dots \subset \ks^t\lg=\lg\]
a \emph{(complex) torus bundle series} with respect to  a complex structure $J$ if for all $i=1\dots , t$
\begin{gather*}
\ks^i\lg \text{ is rational with respect to $\lg_\IQ$ and an ideal in }\ks^{i+1}\lg, \tag{$a$}\\
J\ks^i\lg=\ks^i\lg,\tag{$b$}\\
\ks^{i+1}\lg/\ks^{i}\lg \text{ is abelian  }.\tag{$c$}
\intertext{If in addition}
\ks^{i+1}\lg/\ks^{i}\lg\subset\kz(\lg/\ks^{i}\lg),\tag{$c'$}\label{princ}
\end{gather*}
then $(\ks^i\lg)_{i=0,\dots, t}$ is called a \emph{principal torus bundle series}.

An ascending filtration $(\ks^i\lg)_{i=0,\dots, t}$ on $\lg$ is said to be a  \emph{stable torus bundle series}  for $\lg$, if $(\ks^i\lg)_{i=0,\dots, t}$  is a  torus bundle series for every complex structure $J$ and every rational structure $\lg_\IQ$ in $\lg$. If also condition \eqref{princ} holds then it is called a  \emph{stable principal torus bundle series}.
\end{defin}
Geometrically, a principal torus bundle series induces the holomorphic analogue of the tower of real principal torus bundles described in \eqref{tower}. 

With a  torus bundle series we get in some sense the opposite picture: we start by fibring $M$ over a complex torus with fibre a nilmanifold with left-invariant complex structure and then proceed by decomposing the fibre further. More precisely, the complex structure $J$ restricts to each of the sub-algebras $\ks^i\lg$, and since they are rational we get a nilmanifold with left-invariant complex structure $M_i=\ks^i\Gamma\backslash \ks^iG$ where $\ks^iG=\exp \ks^i\lg$ and $\ks^i\Gamma=\Gamma\cap \ks_iG$. Let $T_i$ be the complex torus associated to $\ks^{i}\lg/\ks^{i-1}\lg$ with the induced complex structure and lattice. The short exact sequences
\[0\to \ks^{i-1}\lg\to\ks^{i}\lg\to \ks^{i}\lg/\ks^{i-1}\lg\to 0\]
give rise to holomorphic fibre bundles
\begin{equation}\label{wall}
{\xymatrix{ M_{i-1} \ar@{^(->}[r] & M_i \ar[d]^{\pi_i}\\ & T_i}} \qquad \text{for } i=1, \dots, t
\end{equation}
with  $M_t=M$ and $M_1=T_1$. Note that these bundles cannot be principal bundles in general since the fibre is not a complex Lie group.

Thus a torus bundle series gives an inductive decomposition of  $M_J$ into complex tori. Considering the complex structure $J_0$ in Example \ref{badex} we see that the length of a (principal) torus bundle series may be larger than the nilpotency index.

The notions of stable (principal) torus bundle series appear to be quite strong but in \cite{rollenske09b} many examples of such have been produced. For example, the classification of complex structures on Lie-algebras with  $\dim \kc^1\lg=1$, worked out independently by several authors,  shows that $0\subset \kz\lg \subset \lg$ is a stable principal torus bundle series \cite[Propostion 3.6]{rollenske09b}.  The notion has the advantage to be independent of the chosen lattice and complex structure and allows to give structural information valid for all nilmanifolds with left-invariant complex structure and Lie-algebra $\lg$.

If we have a holomorphic decomposition as \eqref{tower} on page \pageref{tower} or \eqref {wall} then, by Proposition \ref{indstep}, the inductive approach works and we obtain
\begin{theo}[Console, Fino]
If $M_J$ is a nilmanifold with left-invariant complex structure such that $\lg$ admits a (principal) torus bundle series with respect to $J$ then Conjecture \ref{conj} holds for $M_J$.
\end{theo}
\begin{cor}
 If $\lg$ admits a stable (principal) torus bundle series then Conjecture \ref{conj} holds for every nilmanifold with left-invariant structure with Lie-algebra $\lg$. 
\end{cor}

All possible types of nilmanifolds with left-invariant complex structure up to real dimension 4 were mentioned in the introduction -- there are only complex tori and Kodaira surfaces for which the conjecture is well known.
In real dimension 6 there are only 34  isomorphism classes of nilpotent Lie-algebras and the 18 classes admitting a complex structure have been classified by Salamon \cite{salamon01}. We already met the Lie-algebra $\lh_7$ in Example \ref{badex}. The first part of the following result, which implies the second, is contained in  \cite[Section 4.2]{rollenske09b}.
\begin{cor}
If $M_J$ is a nilmanifold of dimension at most six with Lie-algebra $\lg\ncong\lh_7$ then $\lg$ admits a stable (principal) torus bundle series and Conjecture \ref{conj} holds for $M_J$.
\end{cor}
Roughly half of the Hodge numbers of a nilmanifold $(\Gamma\backslash H_7, J)$ can be checked by hand to coincide with the predictions but the ones in the middle are not immediately accessible.

The conjecture is known to be true in other important special cases. If $M_J$ is the quotient of a complex Lie group, i.e., $(\lg, J)$ is a complex Lie algebra, then the tangent bundle of $M_J$ is holomorphically trivial and $M_J$ is complex parallelisable. This can be reformulated as $[Jx, y]=J[x,y]$ for all $x,y\in \lg$ or equivalently as $[\einsnull\lg, \nulleins\lg]=0$.

Complex structures satisfying the opposite condition $[\einsnull{\lg}, \einsnull\lg]=0$ are called abelian (because $\einsnull \lg$ is an abelian subalgebra of $\lg_\IC$). Such complex structures were introduced by Barberis \cite{barberis99} and come up in different contexts \cite{andrada-salamon05, dotti-fino00}.

In both cases it is straightforward to check that the ascending central series is a principal torus bundle series and thus we have
\begin{cor}\label{acp}
 If $M_J$ is a nilmanifold with left-invariant complex structure and $J$ is abelian or if  $M_J$ is complex parallelisable then $M_J$ is an iterated principal holomorphic torus bundle and Conjecture \ref{conj} holds for $M_J$.
\end{cor}

It was another insight of Console and Fino that the essential issue here is rationality of ideals: consider the descending central series adapted to $J$ defined by
 \[\kc^i_J(\lg)=\kc^i\lg +J\kc^i\lg,\]
in other words $\kc^i_J\lg$ is the smallest $J$-invariant subspace of $\lg$ containing $\kc^i\lg$. Then, by \cite[Lemma 1]{con-fin01}, these subspaces satisfy condition $(b)$ and $(c)$ of Definition \ref{stableseries}. Thus they induce a decomposition of the universal cover $(G,J)$ as an iterated holomorphic bundle over complex vector spaces similar to \eqref{wall}.

The decomposition of the universal cover descends to the compact manifold $M_J$ if and only if  the subspaces $\kc^i_J\lg$  are rational. In particular this is the case, if $J$ itself is rational, i.e., if $J$  maps $\lg_\IQ$ to itself. Thus we have
\begin{cor}[Console,Fino]\label{rational}
 If $J$ is rational then $\lg$ admits a torus bundle series adapted to $J$ and Conjecture \ref{conj} holds for $M_J$.
\end{cor}
This result is very useful, since if one is looking for specific examples usually  everthing can be chosen to be rational.

\subsection{Console and Fino's result on openness}
In the last section we have seen that we can compute Dolbeault cohomology with left-invariant forms whenever we have some control over the geometry of $M_J$. Using deformation theoretic methods one can go further.

Recall that the datum of a complex structure $J:\lg\to \lg$ is equivalent to specifying the subspace $\einsnull\lg\subset \lg_\IC$. So the set of left-invariant complex structures can be identified with the subset 
\[\kc(\lg)=\{ V\in \mathbb{G}r(n , \lg_\IC)\mid V\cap \bar V=0, [ V,V]\subset  V\}\]
of the Grassmannian of half-dimensional subspaces of $\lg_\IC$. The first condition ensures that $\lg_\IC=V\oplus\bar V$ and the second that the complex structure $J_V$ with the corresponding eigenspace decomposition is integrable.

The question when the universal cover decomposes as an iterated principal bundle as in \eqref{tower} has been studied by Cordero, Fern\'andez, Gray and Ugarte. Such left-invariant complex structures are called nilpotent and an algebraic characterisation has been given in \cite{cfgu00}.

Note that it is a hard problem to decide whether $\kc(\lg)\neq\varnothing$ for a given nilpotent Lie-algebra $\lg$.

\begin{theo}[{\cite[Theorem A]{con-fin01}}]\label{open}
Let  $U\subset \kc(\lg)$ be the subset of left-invariant complex structures $J$ for which the inclusion 
\[\phi_J: H^{p,q}(\lg, J)\into H^{p,q}(M_J)\]
is an isomorphism. Then $U$ is an open subset of $\kc(\lg)$.
\end{theo}

The strategy of the proof is to show  that the dimension of the complement of $H^{p,q}(\lg, J)$ in $H^{p,q}(M_J)$ is upper-semi-continuous and thus remains equal to zero in an open  neighbourhood of any point $J$ where $\phi_J$ is an isomorphism.

So to prove Conjecture \ref{conj} it would be sufficient to show that, for each connected component of $\kc(\lg)$, the subset $U$ as in the Theorem is non-empty and closed. Unfortunately Hodge-numbers do behave badly when going to the limit, especially for non-K\"ahler manifolds, so closedness is very difficult.

The set of rational complex structures is a good candidate to show that $U$ is non-empty and dense but it is not clear to me  whether $\kc(\lg)$ does always contain rational complex structures provided it is non-empty. Calculations suggest that this will not be the case but a concrete counterexample is complicated to write down.

\begin{rem}
In Corollary \ref{acp} we saw that the conjecture holds for abelian complex structure and complex parallelisable nilmanifolds. Small deformations of such structures have been studied in some detail and deformations of these are again left-invariant but in general neither abelian nor complex parallelisable (see Section \ref{defos} and \cite{con-fin-poon06, mpps06, rollenske08a}). In this way we can get more examples of interesting complex structures where the conjecture still holds. 
\end{rem}

\subsection{Some new results and open questions}\label{new}
In this section we first present a result that any nilmanifold with left-invariant complex structure is not too far away from satisfying Conjecture \ref{conj}, it suffices to take a finite quotient. This result is new and might lead to a complete proof; we will discuss some possible approaches below.

We first need a lemma that exploits the especially simple arithmetics of lattices in nilpotent Lie groups.
\begin{lem}\label{biglatt}
Let $\lg$ be a nilpotent real Lie algebra, $\Gamma\subset G$ a lattice and $\lg_\IQ$ the rational structure associated to $\log\Gamma$. Then for any $x\in \lg_\IQ$ there exists a lattice $\Gamma'$ such that $\Gamma\subset \Gamma'$ of finite index  and $\exp(x)\in  \Gamma'$. 
\end{lem}
\begin{proof}
Pick any lattice $\tilde\Gamma$ containing $\exp(x)$ and inducing the same rational structure in $\lg$ as $\Gamma$. This is possible by \cite[Lemma 5.1.10]{Cor-Green}. Then by \cite[Theorem 5.1.12]{Cor-Green} $\Gamma\cap\tilde\Gamma$ is a lattice in $G$ which is of finite index in both $\Gamma$ and $\tilde\Gamma$. If we define $\Gamma'$ to be the subgroup of $G$ generated by $\Gamma$ and $\tilde\Gamma$ then $\Gamma'$ is again discrete and contains both $\exp(x)$ and $\Gamma$.\qed
\end{proof}

\begin{prop}\label{quotient}
Let $M_J=(\Gamma\backslash G, J)$ be a nilmanifold with left-invariant complex structure. Then there exists a lattice $\Gamma'\subset G$ with $\Gamma$  of finite index in  $\Gamma'$ such that
\[\phi_J:H^{p,q}(\lg)\isom H^{p,q}(\Gamma'\backslash G,J).\]
\end{prop}
In other word, given any nilmanifold with left-invariant complex structure $M_J$ there is a finite regular covering $\pi:M_J\to M'_J$ such that the conjecture holds for $M_J'$.
\begin{proof}
Endow all involved bundles with left-invariant hermitian metrics. Then the Laplacian $\Delta_{\delbar}=\delbar\delbar^*+\delbar^*\delbar$ is a left-invariant elliptic differential operator on $G$. Let $\kh(G):=\ker(\Delta_{\delbar})$ be the space of harmonic forms of type $(p,q)$ on $G$. We can take invariants under $G$ and $\Gamma$ respectively and get
\[H^{p,q}(M)\isom \kh(G)^\Gamma\supset \kh(G)^G=H^{p,q}(\lg,J).\]
The last equality comes from the compatibility of the Hodge-decomposition with the subspace of left-invariant form; this hat been worked out in detail in \cite{rollenske09a}.

We prove our claim by induction on $d:=\dim \kh(G)^\Gamma-\dim \kh(G)^G$. If $d=0$ we can take $\Gamma'=\Gamma$.

If $d>0$ there exists an $\alpha\in \kh(G)^\Gamma$ and an open subset $U\subset G$ such that
\[g^*\alpha\neq\alpha\]
for $g\in U$.
Let $\lg_\IQ$ be the rational structure induced by $\log(\Gamma)\subset \lg$. Since the exponential map is a diffeomorphism the image of $\lg_\IQ$ is dense in $G$ and we can find an $x\in \lg_\IQ$ such that $\exp (x)\in U$.

By Lemma \ref{biglatt} we can find a lattice $\Gamma'\subset G$ such that $\Gamma\subset \Gamma'$ of finite index and $\exp(x)\in  \Gamma'$; then $\alpha \notin \kh(G)^{\Gamma'}=H^{p,q}( \Gamma'\backslash G, J)$ and we conclude by induction.\qed
\end{proof}

\begin{rem} Proposition \ref{quotient} suggested an approach that unfortunately did not prove successful.
Assume we have constructed for a nilmanifold with left-invariant complex structure $M_J$ a lattice $\Gamma\subset\Gamma'$ as above and then manage to find a way to scale it down, i.e., to find a contracting automorphism $\mu$ of $G$ such that $\mu(\Gamma')=\tilde\Gamma'\subset \Gamma$. This is possible if $\lg$ is naturally graded but not in general \cite{dyer70}. On the level of real manifolds this corresponds to two regular coverings
\[ \tilde M'=\tilde\Gamma'\backslash G\to M \to M'=\Gamma'\backslash G\]
and a (different) isomorphism $\mu:M'\isom\tilde  M'$. 

If $\mu$ preserves the complex structure, i.e., $M'_J$ and $\tilde M'_J$ are isomorphic as complex manifolds then the injections
\[H^{p,q}(\lg, J)=H^{p,q}(M'_J)\into H^{p,q}(M_J)\into H^{p,q}(\tilde M'_J)=H^{p,q}(\lg, J)\]
prove the conjecture for $M_J$. But this will generally not be the case, as can be worked out for the Lie-algebra given in Example \ref{badex}.
\end{rem}

\begin{rem}
We have seen that Conjecture \ref{conj} holds if we understand the complex geometry of a nilmanifold with left-invariant complex structure $M_J$. In addition we have the openness result of Console and Fino. Nevertheless the general case remains open.

There are two other approaches one could try: in the proof of Proposition \ref{quotient} we compared $G$-invariant and $\Gamma$-invariant $\Delta_\delbar$-harmonic differential forms on the universal cover $G$ after choosing some left-invariant hermitian structure. The study of this elliptic operator falls into the realm of harmonic analysis but there does not seem to be a general result that shows that $\Gamma$-invariant harmonic forms are $G$-invariant. One problem is again that $\Delta_\delbar$ does not need to have any compatibility with the natural filtrations on $\lg$ but working on $G$ we might avoid the issue of rationality.

Going back to the compact manifold $M_J$ one might try to use some Weitzenb\"ock formula to express $\Delta_\delbar$ in a different way. But since $M_J$ is in general not K\"ahler the Chern-connection compatible with the hermitian structure will differ from the Levi-Civita connection and again there does not seem to be an applicable general formula at the moment. In this context Gromov's characterisation of nilmanifolds as almost flat manifolds \cite{gromov78} might play an important role.
\end{rem}

\section{Applications}\label{apps}
As mentioned in the introduction, nilmanifolds can be a convenient source of examples in many contexts. Integrability conditions for additional left-invariant geometric structures usually boil down to linear algebra and thus one easily writes down interesting examples of complex, riemannian, hermitian or symplectic structures. Proceeding from the examples to general results is more difficult.

Here we will discuss two further applications related to complex structures references to other areas have already been given in the introduction.

\subsection{Prescribing cohomology behaviour and the Fr\"olicher spectral sequence}
If Conjecture \ref{conj} holds for a nilmanifold with left-invariant complex structure $M_J$ the computation if its Dolbeault cohomology
$H^{p,q}(M_J)=H^{p,q}(\lg, J)$ 
is a matter of finite-dimensional linear algebra and can be taught to a computer algebra system. In addition this makes it possible to study the Fr\"olicher spectral sequence
\[ E_2^{p,q}=H^{p,q}(M_J)\implies H^{p+q}_{dR}(M, \IC),\]
that measures the difference between Dolbeault cohomology and de Rham cohomology. This spectral sequence degenerate at $E_1$ for all compact complex surfaces but Cordero, Fern\'andez, Gray and Ugarte showed in \cite{cfgu99}, studying nilmanifolds,  that for complex 3-folds the maximal non-degeneracy $E_2\ncong E_3=E_\infty$ is possible. Later we constructed a family $X_n\to T_n$ of principal torus bundles over  tori such that $d_n\neq0$ for $X_n$ (see \cite{rollenske07a}). Probably, starting from dimension 3, the maximal non-degeneracy is possible but concrete examples are still missing. If we ask in addition for simply connected manifolds there are only very few examples with non-zero higher differentials known \cite{pittie89}.

The idea behind  these examples is that  if we write down some 1-forms and their differentials carefully enough we get a nilmanifold supporting these forms for free. For example, let $V$, $W$ be two complex vector spaces  and give an arbitrary map
\[\delta: W^* \to \Lambda^2 V^*\tensor ( V^*\tensor \bar V^*).\]
Setting $\einsnull\lg=V\oplus W$ and $\lg_\IC:=\einsnull \lg\oplus \overline{\einsnull \lg}$ we extend $\delta$ to a map 
\[d:\lg^*_\IC\to \Lambda^2\lg^*_\IC\]
 which is zero on $V^*\oplus \bar V^*$ and $\delta+\bar\delta$ on $W^*\oplus \bar W^*$.
There is a natural real vector space $\lg=\{z+\bar z\mid z\in \einsnull{\lg}\}\subset \lg_\IC$ and via the identity 
\[ d\alpha(x,y)=-\alpha([x,y])\quad \text{for $\alpha\in \lg^*$ and $x,y\in \lg$}\]
the vector space $\lg$ becomes a 2-step nilpotent Lie-algebra. The decomposition of $\lg_\IC$ defines an almost complex structure $J$ on $\lg$ which is integrable by our choice that $\delta$ has no component mapping to $\Lambda^2\bar V$. If we have chosen $\delta$ such that the structure constants of $\lg$ turn out to be rational there exists a lattice in the associated nilpotent Lie-group and we have constructed a nilmanifold $M_J$ with left-invariant complex structure. 

Nearly by definition $M_J$ is a principal holomorphic torus bundle over a torus and thus we not only have prescribed the differential of some 1-forms quite arbitrarily but our datum encodes in fact the whole cohomology algebra.

Constructing nilmanifolds with higher nilpotency index in a similar way is more tedious since one has to take care of the Jacobi identity, equivalent to $d^2=0$, as well.

\subsection{Deformations of complex structures}\label{defos}
Our main motivation to study Conjecture \ref{conj} was the question if small deformations of  left-invariant complex structures remain left-invariant. Generalising results of Console, Fino and Poon \cite{con-fin-poon06} (see also \cite{mpps06}) we proved

\begin{theo}[{\cite[Theorem 2.6]{rollenske09a}}]
 If Conjecture \ref{conj} holds for a nilmanifold with left-invariant complex structure $M_J$ then all sufficiently small deformations of $J$ are again left-invariant complex structures.
\end{theo}
The idea of the proof is that small deformations of $J$ are controlled by the first and second cohomology groups of the holomorphic tangent bundle. By constructing a version of Serre-duality that works purely on the level of Lie-algebra cohomology one can represent the elements of $H^i(M_J, \mathcal{T}_{M_J})$ by left-invariant forms and the result follows by the standard inductive construction of the Kuranishi space \cite{kuranishi62}.

The space of all integrable complex structures on a nilmanifold $M$ modulo orientation preserving diffeomorphisms isotopic to the identity is called Teichm\"uller space  $\mathfrak{T}(M)$. It is (locally) a complex analytic space, the germ at a fixed complex structure $J$ being the Kuranishi space of $(M,J)$. Thus the theorem says that, under the assumption of Conjecture \ref{conj}, the set of left-invariant complex structures is open in $\mathfrak{T}(M)$.

If the Lie algebra $\lg$ of $M$ admits a stable (principal) torus bundle series (see Definition \ref{stableseries}) then Conjecture \ref{conj} holds for all left-invariant complex structures on $\lg$ and it is natural to ask if the set of left-invariant complex structures is also closed. The starting point  in this direction is Catanese's result that all deformations in the large of a complex torus are complex tori \cite{catanese02}. Generalising results of Catanese and Frediani \cite{catanese04, cat-fred06} this was extended in \cite{rollenske09b} to a large class of nilmanifolds with left-invariant complex structure. As an example we would like to mention that every deformation in the large of the Iwasawa manifold is a nilmanifold with left-invariant complex structure; in this case the topology of the space of left-invariant complex structures is known \cite{ket-sal04}.

In this area many interesting questions remain open, we hope to address some of these in future work. Progress in the direction of Conjecture \ref{conj} would encourage our belief that the complex geometry of nilmanifolds with left-invariant complex structure can be completely understood via linear algebra.

\subsection*{acknowledgement}
We would like to thank the organisers for the stimulating conference and the invitation to contribute to this volume. We enjoyed several discussion on this topic with Fabrizio Catanese and Uwe Semmelmann. Anna Fino provided some interesting references to the literature. Careful remarks by the referee helped to improve the presentation.
During the preparation of this article the author was supported by the Hausdorff Centre for Mathematics in Bonn.

\end{document}